\documentclass[journal,twoside,web,letterpaper]{ieeecolor} 
\let\labelindent\relax

\usepackage{lcsys}
\usepackage{cite}
\usepackage{amsmath,amssymb,amsthm,upgreek,graphics,epsfig,times,enumitem,
	float,mathtools,mathrsfs,bm}
\usepackage{hyperref} 
\hypersetup{colorlinks=true, linkcolor=blue, filecolor=blue, urlcolor=blue, citecolor=brown}
\usepackage[linesnumbered,ruled]{algorithm2e}
\usepackage[caption=false]{subfig} 
\usepackage[dvipsnames]{xcolor}
\usepackage{times,mathptmx}
\usepackage{authblk}
\usepackage{tikz}
\usetikzlibrary{arrows}
\usetikzlibrary{shapes}
\usetikzlibrary{arrows.meta} 
\usetikzlibrary{calc,intersections,through,backgrounds}
\usetikzlibrary{positioning}
\usepackage{pgfplots}
\pgfplotsset{compat=1.16}
\usepackage{pgfplotstable}
\usepgfplotslibrary{groupplots}
\usepackage{filecontents}
\theoremstyle{theorem}
\newtheorem{theorem}{Theorem}

\theoremstyle{definition}
\newtheorem{definition}{Definition}
\newtheorem{assumption}{Assumption}
\theoremstyle{remark}
\newtheorem{remark}{Remark}
 
\newcommand{\R}{\mathbb{R}}
\newcommand{\N}{\mathsf{N}}
\newcommand{\K}{\mathsf{K}}

\newcommand{\U}{\mathsf{U}}
\newcommand{\E}{\mathbb{E}}
\newcommand{\bu}{\mathsf{u}}

\newcommand{\col}[1]{\mathrm{col}(#1)}  
\newcommand{\row}[1]{\mathrm{row}(#1)}  

\usepackage{textcomp}
\def\BibTeX{{\rm B\kern-.05em{\sc i\kern-.025em b}\kern-.08em
    T\kern-.1667em\lower.7ex\hbox{E}\kern-.125emX}}
\begin{document}

\title{Linear-quadratic mean-field-type difference games with coupled affine inequality constraints}
\author{Partha Sarathi Mohapatra and Puduru Viswanadha Reddy, \IEEEmembership{Member, IEEE} 
		\thanks{P. S. Mohapatra and P. V. Reddy are with the Department of Electrical Engineering, Indian Institute of Technology-Madras, Chennai, 600036, India.
		{(e-mail: ps\textunderscore mohapatra@outlook.com, vishwa@ee.iitm.ac.in)}}} 	
	\date{ }
	\maketitle
        \thispagestyle{headings}
	
\begin{abstract}
	 In this letter, we study a class of linear-quadratic mean-field-type difference games with coupled affine inequality constraints. We show that the mean-field-type equilibrium  
	 can be characterized by the existence of a multiplier process which satisfies
	 some implicit complementarity conditions. Further, we show that the equilibrium strategies can be computed by reformulating these conditions as a single large-scale linear complementarity problem. 
	 We illustrate our results with an energy storage problem arising in the management of microgrids.
\end{abstract}

\section{Introduction}
Stochastic dynamic games offer a mathematical framework for modeling dynamic decision-making scenarios that involve multiple players interacting in uncertain environments. Mean-field-type dynamic games (MFTDGs) are a specific class of stochastic games that allow for the inclusion of not just the state and control terms, but also their distributions in the objective functionals and state dynamics \cite{Tembine:17a,Barreiro:21}. As a result, MFTDGs, when the mean and variance terms are considered, are related to the mean-variance paradigm developed by H. Markowitz \cite{Markowitz:52}. MFTDGs differ from mean-field games \cite{Caines:18} by accounting for inherent heterogeneities and finite decision-makers, unlike the approximation provided by mean-field games for problems with many symmetric players. MFTDGs have been increasingly utilized to model real-world engineering problems arising in water networks \cite{BarreiroA:19},  smart-grids\cite{Djehiche:18}  and pedestrian flow \cite{Aurell:18}; see \cite{Barreiro:21} and \cite{Bensoussan1:13} for a detailed coverage.

MFTDGs have been solved using various approaches in the literature, including the stochastic maximum principle \cite{Djehiche:15}, dynamic programming \cite{Pham:16}, and direct method \cite{Duncan:18, Barreiro:B19}. In \cite{Gomez:19a}, the authors studied MFTDGs involving equality constraints on the control variables.
Multi-agent decision problems in engineering and economics require incorporating inequality constraints on state and control variables like capacity, saturation, and budget constraints. The dynamic nature of these constraints poses technical challenges in characterizing admissible controls and establishing solvability conditions for equilibria, distinguishing them from unconstrained counterparts. Despite the significance of these challenges, the literature on MFTDGs involving inequality constraints is scarce, to the best of our knowledge.

This letter proposes a solution for linear-quadratic MFTDGs with inequality constraints. We focus on a discrete-time setting with finite horizon, scalar state dynamics, and quadratic objectives. Our approach incorporates coupled affine inequality constraints on the mean values of state and control variables. Using the direct method, also known as completion of squares, we establish a connection between the existence of a solution for these MFTDGs and the existence of a multiplier process satisfying implicit complementarity conditions. By leveraging an approach similar to \cite{Reddy:15}, we transform these existence conditions into the solvability of a single large-scale linear complementarity problem, thereby providing a computational method for solving these games. The proposed approach can be easily extended to matrix-valued settings.

The letter is organized as follows. In Section \ref{sec:Preliminaries}, we introduce a class of MFTDGs with coupled affine inequality constraints. In Section \ref{sec:MFTNE}, we provide conditions for the existence of  an  equilibrium in these games. In Section \ref{sec:Solvability}, we present an approach for reformulating these conditions as a linear complementarity problem. In Section  \ref{sec:Numerical}, we illustrate our method with numerical simulations and finally   Section \ref{sec:Conclusion} concludes.
\subsubsection*{Notation}  We denote  the transpose of any vector $a$ or matrix $A$ by $a^{\prime}$ and $A^{\prime}$ respectively. The   identity matrix and the matrix of zeros are represented by $\mathbf{I}$ and $\mathbf{0}$    respectively, with appropriate dimensions;  unless specified, is determined from the context. $\E[x]$ denotes the expected value of $x$. Let $A \in \R^{n\times n}$ be partitioned as $n = n_1+ \cdots+n_K$. We represent $[A]_{ij}$ as the $n_i \times n_j$ sub-matrix associated with indices $n_i$ (row) and $n_j$ (column). We denote the column vector $[v_1^{\prime}, \cdots, v_n^{\prime}]^{\prime}$  by $\col{v_k}_{k=1}^n$ and the row vector $[v_1~ \cdots~ v_n]$ by   $\row{v_k}_{k=1}^n$. The block diagonal matrix obtained by taking the matrices $M_1, \cdots, M_K$ as diagonal elements in this sequence, is   represented by $\oplus_{k=1}^{K}M_k$. We represent the Kronecker product operation by $\otimes$. We call two vectors $x, y \in \R^n$ complementary if $x \geq 0,~ y \geq 0$ and $x^{\prime}y=0$, and we compactly denote these conditions by $0 \leq x \perp y \geq 0$.  


\section{Preliminaries}\label{sec:Preliminaries}
In this section we introduce  a class of $N$-player scalar finite-horizon mean-field-type difference game with inequality constraints (MFTDGC).  We denote the set of players by $\N = \{1, 2, \cdots, N\}$, the set of time instants or decision stages by $\K = \{0, 1, ..., K\}$. We define  the following two sets as $\K_l:=\K\setminus \{K\}$ and $\K_r:=\K\setminus \{0\}$.  At each time instant $k \in \K_l$, each player $i \in \N$ chooses an action $u^i_k\in \R $ and influences the evolution of state variable $x_k\in \R$ according to the following discrete-time linear dynamics 
\begin{subequations}\label{eq:MFgame}
	\begin{align}
		x_{k+1} &= a_k x_k+\bar{a}_k \E[x_k]+\sum_{i\in\N}(b^i_k u_k^i+\bar{b}^i_k \E[u_k^i])+c_k+\sigma_kw_k, \label{eq:MFstate}
	\end{align}
	where the initial state $x_0$ is a scalar random variable  with known finite mean and known finite variance, and $a_k$, $\bar{a}_k$, $b_k^i$, $\bar{b}_k^i$, $c_k$, $\sigma_k\in\R,~ i \in \N$. $w_k\in\R$ denotes a stochastic disturbance with zero-mean and finite variance. We assume that the decisions of each player $i\in\N$ additionally satisfy the following mixed affine coupled inequality constraints 
	\begin{align}
		\bar{m}_k^i\E[x_k]+\sum_{j\in\N}\bar{n}^{ij}_k \E[u_k^j]+p_k^i \geq 0, \label{eq:constraints}
	\end{align}
	where $\bar{m}_k^i$, $\bar{n}_k^{ij}$, $p_k^i\in \R^{s_i}$, $k\in \K_l$. We denote the set of players excluding the player $i$ by $-i:=N\setminus \{i\}$. At any instant $k\in \K_l$ the collection of actions of all players excluding player $i$ is denoted by $u^{-i}_k:=\col{u_k^1,\cdots,u_k^{i-1},u_k^{i+1},\cdots,u_k^N}$. The profile of actions, also referred to as a strategy, of player $i\in \N$ is denoted by $\bu^{i}:=\col{u_k^{i}}_{k=0}^{K-1}$, and the strategy  of all players except player $i$ is denoted by $\bu^{-i}:=\col{u_k^{-i}}_{k=0}^{K-1}$. Each player $i\in\N$ while choosing their actions seeks to minimize the following interdependent stage additive cost functional
	\begin{align}\label{eq:objective1}
		&J^i(\bu^i, \bu^{-i}) =\tfrac{1}{2}
		q_K^i(x_K)^2+\tfrac{1}{2}\bar{q}_K^i\E[x_K]^2+\tfrac{1}{2}\sum_{k\in\K_l}\big(q_k^i(x_k)^2\nonumber\\
		&\hspace{0.2in}+\bar{q}_k^i\E[x_k]^2\big)+\tfrac{1}{2}\sum_{k\in\K_l}\sum_{j\in\N}\big(r_k^{ij}(u_k^j)^2+\bar{r}_k^{ij}\E[u_k^j]^2\big),
	\end{align}
\end{subequations}
where  $q_k^i,\bar{q}_k^i\in\R$, $k\in\K$ and $r_k^{ij},\bar{r}_k^{ij}\in\R$, $k\in\K_l$.  
\begin{remark} From \eqref{eq:MFstate}, the expected state dynamics is given by
	\begin{align}
		\E[x_{k+1}] &= (a_k+\bar{a}_k) \E[x_k]+\sum_{i\in\N}(b_k^i+\bar{b}^i_k) \E[u_k^i]+c_k. \label{eq:MFPmeanstate1}
	\end{align}
	As $\E[\E[x_k](x_k-\E[x_k])]=0$ and $\E[\E[u_k^j](u_k^j-\E[u_k^j])]=0$, the expected objective can also be represented with mean and variance terms of state and control variables as follows
	\begin{align}
		&\E[J^i(\bu^i, \bu^{-i})] =
		\tfrac{1}{2}q_K^i\E[\big(x_K-\E[x_K]\big)^2]+\tfrac{1}{2}(q_K^i+\bar{q}_K^i)\E[x_K]^2\nonumber\\
		&\quad+\tfrac{1}{2}\sum_{k\in\K_l}\big(q_k^i\E[(x_k-\E[x_k])^2]+(q_k^i+\bar{q}_k^i)\E[x_k]^2\big)\nonumber\\
		&\quad+\tfrac{1}{2}\sum_{k\in\K_l}\sum_{j\in\N}\big(r_k^{ij}\E[(u_k^j-\E[u_k^j])^2]+(r_k^{ij}+\bar{r}_k^{ij})\E[u_k^j]^2\big).\label{eq:meanobjective}
	\end{align} 
\end{remark}
The constraints  given by \eqref{eq:constraints} are coupled i.e., at every stage $k\in \K_l$, the control actions $u_k^{-i}$ of players in $-i$ impose a restriction   on   player $i$'s control action $u_k^i$. Collecting  the constraints of all the players, and by eliminating the expectation of state variable using \eqref{eq:MFPmeanstate1}, we get  
\begin{multline}
	\tilde{\mathsf{M}}_k \big((a_{k-1}+\bar{a}_{k-1})\cdots (a_{0}+\bar{a}_{0}) \E[x_0]+(a_{k-1}+\bar{a}_{k-1})\cdots \\
	\times(a_{1}+\bar{a}_{1}) \bar{\mathsf{B}}_{0}\E[\bu_{0}]+\cdots+ (a_{k-1}+\bar{a}_{k-1})\bar{\mathsf{B}}_{k-2}\E[\bu_{k-2}]\\
	+ \bar{\mathsf{B}}_{k-1}\E[\bu_{k-1}]\big)
	+\bar{\mathsf{N}}_k \E[\bu_k]+\bm{p}_k \geq 0,\label{eq:vectorC1}
\end{multline}
where $\tilde{\mathsf{M}}_k:=\col{\bar{m}_k^i}_{i=1}^{N}$,  $\bar{\mathsf{N}}_k:=\col{\row{\bar{n}_k^{ij}}_{j=1}^{N}}_{i=1}^{N}$, $\bar{\mathsf{B}}_k:=\row{b_k^i+\bar{b}_k^i}_{i=1}^{N}$, $\E[\bu_k]=\col{\E[u_k^i]}_{i=1}^{N}$ and $\bm{p}_k=\col{p_k^i}_{i=1}^N$. The joint feasible strategy space of the players is given by
\begin{align}
	R(\E[x_0]):=\{&(\bu^{i}, \bu^{-i})\in \R^{KN}:\eqref{eq:vectorC1}~\textrm{holds } \forall k \in \K_l\}.
	\label{eq:CConstraint}
\end{align}  
Using \eqref{eq:CConstraint}, the admissible strategy space of player $i\in\N$ for  a given $\E[x_0]\in \mathbb R$ and $\bu^{-i}$ is given by 
\begin{align}
	\U^i(\bu^{-i}):=\{\bu^{i}\in\R^{K}:(\bu^{i}, \bu^{-i}) \in R(\E[x_0])\}.\label{eq:AdmissibleSet}
\end{align}
Next, we have the following assumption.
\begin{assumption}\label{ass:GameAssumption}
	\begin{enumerate}[label = (\roman*)]
		\item \label{item:1} For a given $\E[x_0]\in \mathbb R$, the joint admissible strategy set $R(\E[x_0])\subseteq \mathbb R^{KN}$ is non-empty.
		\item \label{item:2} All the elements of the vector $\bar{n}^{ii}_k\in \mathbb{R}^{s_i}$ are non-zero for all $k\in\K_l$ and $i\in\N$.
		\item \label{item:3} For each player $i\in\N$, $q_k^i$, $q_k^i+\bar{q}_k^i \geq 0$, $k\in\K$ and $r_k^{ii}$, $r_k^{ii}+\bar{r}_k^{ii} > 0$, $k\in\K_{l}$.
	\end{enumerate}
\end{assumption}
Item \ref{item:1} is required to guarantee the existence of a solution of  \eqref{eq:MFstate} satisfying \eqref{eq:constraints}, for a given  $\E[x_0]\in\R $. Item \ref{item:2} is required to satisfy the constraint qualification conditions.
Item \ref{item:3} is a technical assumption which can be relaxed; see Remark \ref{rem:relaxedCond}.

The non-cooperative outcome, that is, mean-field-type Nash equilibrium associated with     MFTDGC, described by \eqref{eq:MFgame}, is defined as follows.
\begin{definition}\label{def:MFTNEdef} 
	For a given $\E[x_0]\in\R$, an admissible strategy profile $( \bu^{i\star}, \bu^{-i\star}) \in R(\E[x_0])$ is a mean-field-type generalized Nash equilibrium (MFTGNE) for   MFTDGC, if for each player $i\in \N$ the following condition holds
	\begin{align}
		\E[J_i( \bu^{i\star}, \bu^{-i\star})] \leq \E[J_i(\bu^i, \bu^{-i\star})],~\forall \bu^i\in \U^i(\bu^{-i\star}).\label{eq:MFTNEdef}
	\end{align} 
\end{definition}
In this letter, we seek to obtain conditions for the existence of MFTGNE for MFTDGC.
\section{Main result}\label{sec:MFTNE} 
In this section, we present a characterization of MFTGNE for MFTDG. To this end, we employ the direct method, which involves a five-step procedure for finding the solution.
\begin{theorem}\label{thm:MFT_NE}  Let Assumption \ref{ass:GameAssumption} holds. Assume  there exist a multiplier process $\{\mu_k^{i\star}\in \R^{s_i},~i\in \N,~ k\in\K_l\}$ satisfying the following complementarity conditions
\begin{align}
    0\leq  \big(\bar{m}_k^i+\sum_{j\in \N}\bar{n}_k^{ij}\delta_k^j\big)\E[x_k^{\star}]+\sum_{j\in \N}\bar{n}_k^{ij}\bar{\delta}_k^j+p_k^i \perp \mu_k^{i\star} \geq 0,\label{eq:KKTcomp}
\end{align}
where $\{x_k^{\star}, ~k\in\K_l\}$ evolves as follows
\begin{subequations}\label{eq:Eqstate} 
\begin{align}
    &x_{k+1}^{\star}-\E[x_{k+1}^{\star}]=\mathsf{A}_k(x_k^{\star}-\E[x_k^{\star}])+\sigma_k w_k,~x_0^{\star}=x_0,\label{eq:equistatediff}\\
    &\E[x_{k+1}^{\star}] = \bar{\mathsf{A}}_k\E[x_k^{\star}]+\big(\sum_{j\in\N}(b_k^j+\bar{b}^j_k) \bar{\delta}_k^j+c_k\big),\label{eq:equistate}
\end{align}
\end{subequations}
with $\mathsf{A}_k:=a_k+\sum_{j\in\N}b^j_k\eta_k^j$, $\bar{\mathsf{A}}_k:=\big(a_k+\bar{a}_k+\sum_{j\in\N}(b_k^j+\bar{b}^j_k)\delta_k^j \big)$ and for each $i\in\N$, $\{  \eta_k^i ,  \delta_k^i ,\bar{\delta}_k^i,~k\in \K_l \}$ satisfy the following algebraic equations
\begin{subequations}\label{eq:coefficients}
    \begin{align}
        &r_k^{ii}\eta_k^i+b_k^i\alpha_{k+1}^i\sum_{j\in \N}b^j_k\eta_k^j+b^i_k\alpha_{k+1}^i a_k =0,\\
        &(r_k^{ii}+\bar{r}_k^{ii})\delta_k^i+(b_k^i+\bar{b}^i_k)\bar{\alpha}_{k+1}^i\sum_{j\in \N}(b_k^j+\bar{b}^j_k)\delta_k^j\nonumber\\
        &\qquad \qquad+(b_k^i+\bar{b}^i_k)\bar{\alpha}_{k+1}^i(a_k+\bar{a}_k)=0,\\
        &(r_k^{ii}+\bar{r}_k^{ii})\bar{\delta}_k^i+(b_k^i+\bar{b}^i_k)\bar{\alpha}_{k+1}^i\sum_{j\in \N}(b_k^j+\bar{b}^j_k)\bar{\delta}_k^j\nonumber\\
        &\qquad \qquad+(b_k^i+\bar{b}^i_k)\big(\bar{\alpha}_{k+1}^ic_k+\beta_{k+1}^i\big)-\bar{n}_k^{ii}{}^{\prime}\mu_k^{i\star}=0, \label{eq:bdeltaeq}
    \end{align}
\end{subequations}
where $\alpha_k^i$, $\bar{\alpha}_k^i$ and $\beta_k^i$ for $k\in \K_l$ are obtained by solving the following backward difference equations
\begin{subequations}\label{eq:BackwardRecursive}
    \begin{align}
      \alpha_{k}^i&=\alpha_{k+1}^i(\mathsf{A}_k)^2+\sum_{j\in \N}r_k^{ij}(\eta_k^j)^2+q_k^i,\label{eq:alpharecursive}\\
      \bar{\alpha}_{k}^i&=\bar{\alpha}_{k+1}^i(\bar{\mathsf{A}}_k)^2+\sum_{j\in \N}(r_k^{ij}+\bar{r}_k^{ij})(\delta_k^j)^2+(q_k^i+\bar{q}_k^i),\label{eq:baralpharecursive}\\
      \beta_k^i&=\bar{\mathsf{A}}_k\beta_{k+1}^i+\sum_{j\in -i}\big((r_k^{ij}+\bar{r}_k^{ij})\delta_k^j+\bar{\mathsf{A}}_k\bar{\alpha}_{k+1}^i(b_k^j+\bar{b}^j_k)\big)\bar{\delta}_k^j\nonumber\\
      &~~-\big(\bar{m}_k^i+\sum_{j\in\N}\bar{n}_k^{ij}\delta_k^j\big)^{\prime}\mu_k^{i\star}+\bar{\mathsf{A}}_k\bar{\alpha}_{k+1}^ic_k,\label{eq:Betarecursive}
    \end{align}
\end{subequations}
with boundary conditions $\alpha_K^i=q_K^i$, $\bar{\alpha}_K^i=q_K^i+\bar{q}_K^i$ and $\beta_K^i=0$. Then, the MFTGNE strategy of each player $i\in\N$ is given by
\begin{subequations}\label{eq:MFT_NE}
    \begin{align}
        u_k^{i\star}-\E[u_k^{i\star}]&=\eta_k^i(x_k^{\star}-\E[x_k^{\star}]),\\
    \E[u_k^{i\star}]&=\delta_k^i\E[x_k^{\star}]+\bar{\delta}_k^i.
    \end{align}
\end{subequations}  
Furthermore, the expected equilibrium cost of player $i\in \N$ is given by
$\E[J^i(\bu^{i\star}, \bu^{-i\star})] =\tfrac{1}{2}\alpha_0^i\E[(x_0-\E[x_0])^2]+\tfrac{1}{2}\bar{\alpha}_0^i\E[x_0]^2+\beta_0^i\E[x_0]+\gamma_0^i$, 
where $\gamma_k^i$ for $k\in\K_l$ is obtained from the following backward difference equation with boundary condition $\gamma_K^i=0$.
\begin{align}
    \gamma_k^i&=\gamma_{k+1}^i-\tfrac{1}{2}\beta_{k+1}^i(b_k^i+\bar{b}^i_k) \bar{\delta}_k^i+\tfrac{1}{2}\sum_{j\in -i}(r_k^{ij}+\bar{r}_k^{ij})(\bar{\delta}_k^j)^2\nonumber\\
      &\quad+\tfrac{1}{2}\big(\sum_{j\in \N}(b_k^j+\bar{b}^j_k) \bar{\delta}_k^j+c_k\big)\big(\bar{\alpha}_{k+1}^i\big(\sum_{j\in-i}(b_k^j+\bar{b}^j_k) \bar{\delta}_k^j+c_k\big)\nonumber\\
      &\quad+2\beta_{k+1}^i\big)-\tfrac{1}{2}{\mu_k^{i\star}}^{\prime}\big(2\sum_{j\in \N}\bar{n}_k^{ij}\bar{\delta}_k^j-\bar{n}_k^{ii}\bar{\delta}_k^i+2p_k^i\big)\nonumber\\
      &\quad+\tfrac{1}{2}\alpha_{k+1}^i(\sigma_k)^2 \E[(w_k)^2].\label{eq:gammarecursive} 
\end{align}
\end{theorem}
\begin{proof} 
First, it is straightforward to verify that when all the players use   strategies, given by \eqref{eq:MFT_NE}, then $\{x^\star_k,~k\in \K\}$, given by \eqref{eq:Eqstate}, is the generated state trajectory. Due to \eqref{eq:KKTcomp}, the inequality constraints \eqref{eq:constraints} hold for all players with  the strategy profile $(\bu^{i\star},\bu^{-i\star})$. This implies, $(\bu^{i\star},\bu^{-i\star})\in R(\E[x_0])$ or $\bu^{i\star}\in\U^i(\bu^{-i\star})$, and  in particular, 	$ \U^i(\bu^{-i\star})\neq \emptyset$. 
	Consider any admissible strategy profile $ \bu^i\in \U^i(\bu^{-i\star})$, and let $\{x_k,~k\in\K\}$ be the corresponding state trajectory. Using \eqref{eq:MFT_NE} for all players in $-i$, in \eqref{eq:MFstate}, \eqref{eq:MFPmeanstate1}, and \eqref{eq:constraints} we obtain the following relations
	\begin{subequations} \label{eq:allexp}
\begin{align}
 &   \E[x_{k+1}]  = \big(a_k+\bar{a}_k+\sum_{j\in-i}(b_k^j+\bar{b}^j_k)\delta_k^j \big) \E[x_k]\notag\\
    &\quad+(b_k^i+\bar{b}^i_k) \E[u_k^i]+\big(\sum_{j\in-i}(b_k^j+\bar{b}^j_k) \bar{\delta}_k^j+c_k\big), \label{eq:expr1}\\
  &  \E[x_{k+1}]^2   = \big(a_k+\bar{a}_k+\sum_{j\in-i}(b_k^j+\bar{b}^j_k)\delta_k^j \big)^2 \E[x_k]^2\notag\\
    &\quad+(b_k^i+\bar{b}^i_k)^2 \E[u_k^i]^2+\big(\sum_{j\in-i}(b_k^j+\bar{b}^j_k) \bar{\delta}_k^j+c_k\big)^2\notag\\
    &\quad+2\big(a_k+\bar{a}_k+\sum_{j\in-i}(b_k^j+\bar{b}^j_k)\delta_k^j \big)(b_k^i+\bar{b}^i_k) \E[x_k]\E[u_k^i]\notag\\
    &\quad +2\big(a_k+\bar{a}_k+\sum_{j\in-i}(b_k^j+\bar{b}^j_k)\delta_k^j \big)\big(\sum_{j\in-i}(b_k^j+\bar{b}^j_k) \bar{\delta}_k^j+c_k\big)\notag\\
    &\quad\times\E[x_k]+2(b_k^i+\bar{b}^i_k)\big(\sum_{j\in-i}(b_k^j+\bar{b}^j_k) \bar{\delta}_k^j+c_k\big)\E[u_k^i],\label{eq:expr2}\\
    &\E[\big(x_{k+1}-\E[x_{k+1}]\big)^2]=\big(a_k+\sum_{j\in-i}b^j_k\eta_k^j\big)^2\E[(x_k-\E[x_k])^2]\notag\\
    &\quad+(b^i_k)^2 \E[(u_k^i-\E[u_k^i])^2]+(\sigma_k)^2 \E[(w_k)^2] \notag
\end{align}
\begin{align}
    &\quad+2\big(a_k+\sum_{j\in-i}b^j_k\eta_k^j\big)b^i_k\E[(x_k-\E[x_k])(u_k^i-\E[u_k^i]), \label{eq:expr3}\\
    &\big(\bar{m}_k^i+\sum_{j\in -i}\bar{n}_k^{ij}\delta_k^j\big)\E[x_k]+\bar{n}^{ii}_k \E[u_k^i]+\sum_{j\in -i}\bar{n}_k^{ij}\bar{\delta}_k^j+p_k^i \geq 0. \label{eq:consubstituted}
\end{align}
\end{subequations}
 Next we use the direct method to complete the proof. \\
 \noindent 
\textbf{Step 1 -- (Defining a guess functional)}:
We first define a quadratic guess functional of the following form
\begin{align*}
     f^i(k,x_k)=\tfrac{1}{2}\alpha_k^i(x_k-\E[x_k])^2+\tfrac{1}{2}\bar{\alpha}_k^i\E[x_k]^2+\beta_k^i\E[x_k]+\gamma_k^i. 
\end{align*}
\noindent 
\textbf{Step 2 -- (Telescopic sum of the guess functional):} 
Upon taking the telescopic sum of $f^i(k,x_k)$ over $k\in \K$ we obtain
\begin{align} 
	f^i(0,x_0)&=f^i(K,x_K)-\sum_{k\in\K_l}\big(f^i(k+1,x_{k+1})-f^i(k,x_k)\big)\nonumber\\
   &=\tfrac{1}{2}\alpha_K^i(x_K-\E[x_K])^2+\tfrac{1}{2}\bar{\alpha}_K^i\E[x_K]^2+\beta_K^i\E[x_K]+\gamma_K^i\nonumber\\
   &~~-\tfrac{1}{2}\sum_{k\in\K_l}\big(\alpha_{k+1}^i(x_{k+1}-\E[x_{k+1}])^2-\alpha_k^i(x_k-\E[x_k])^2\big)\nonumber\\
   &~~-\tfrac{1}{2}\sum_{k\in\K_l}\big(\bar{\alpha}_{k+1}^i\E[x_{k+1}]^2-\bar{\alpha}_k^i\E[x_k]^2\big)\nonumber\\
   &~~-\tfrac{1}{2}\sum_{k\in\K_l}\big(\beta_{k+1}^i\E[x_{k+1}]-\beta_k^i\E[x_k]+\gamma_{k+1}^i-\gamma_k^i\big).
	\label{eq:telescopicsum}
\end{align} 
 \noindent 
\textbf{Step 3 -- (Difference between the cost and  the guess functional):}
Next, using the expressions \eqref{eq:meanobjective}, \eqref{eq:expr1}-\eqref{eq:expr3} and \eqref{eq:telescopicsum} we compute    $\E[J^i(\bu^{i}, \bu^{-i})-f^i(0,x_0)]$ as follows
 \begin{align}
 	&\E[J^i(\bu^i, \bu^{-i\star})-f^i(0,x_0)]=\tfrac{1}{2}(q_K^i-\alpha_K^i)\E[\big(x_K-\E[x_K]\big)^2]\nonumber\\
 	&+\tfrac{1}{2}(q_K^i+\bar{q}_K^i-\bar{\alpha}_K^i)\E[x_K]^2-\beta_K^i\E[x_K]-\gamma_K^i\nonumber\\
 	&+\tfrac{1}{2}\sum_{k\in\K_l}\big((C_k^i+\tfrac{(B^i_k)^2}{A_k^i})\E[(x_k-\E[x_k])^2]+(L_k^i+\tfrac{(F_k^i)^2}{D_k^i})\E[x_k]^2\nonumber\\
 	&+\sum_{k\in\K_l}\big(M_k^i+{\mu_k^{i\star}}^{\prime}\big(\bar{m}_k^i+\sum_{j\in -i}\bar{n}_k^{ij}\delta_k^j\big)+\tfrac{F_k^iG_k^i}{D_k^i}\big)E[x_k]\nonumber\\
 	&+\sum_{k\in\K_l}\big(N_k^i+{\mu_k^{i\star}}^{\prime}\big(\sum_{j\in -i}\bar{n}_k^{ij}\bar{\delta}_k^j+p_k^i\big)+\tfrac{1}{2}\tfrac{(G_k^i)^2}{D_k^i}\big)\nonumber\\
 	&+\tfrac{1}{2}\sum_{k\in\K_l}\big(A_k^i\E[(u_k^i-\E[u_k^i])^2]+2B_k^i\E[(x_k-\E[x_k])(u_k^i-\E[u_k^i])]\big)\nonumber\\
 	&+\tfrac{1}{2}\sum_{k\in\K_l}\big(D_k^i\E[u_k^i]^2+2\big(F_k^i\E[x_k]+G_k^i+{\mu_k^{i\star}}^{\prime}\bar{n}_k^{ii}\big)\E[u_k^i]\big),\label{eq:PiLagrangian1}
 \end{align}
where $A_k^i :=r_k^{ii}+\alpha_{k+1}^i(b^i_k)^2$, $
B_k^i:=\alpha_{k+1}^i\big(a_k+\sum_{j\in-i}b^j_k\eta_k^j\big)b^i_k$,
\begin{align*}
    C_k^i&:=\big(q_k^i+\sum_{j\in -i}r_k^{ij}(\eta_k^j)^2\big)+\alpha_{k+1}^i\big(a_k+\sum_{j\in-i}b^j_k\eta_k^j\big)^2 -\alpha_k^i-\tfrac{(B^i_k)^2}{A_k^i},\\
    D_k^i&:=r_k^{ii}+\bar{r}_k^{ii}+\bar{\alpha}_{k+1}^i(b_k^i+\bar{b}^i_k)^2,\\
    F_k^i&:=\bar{\alpha}_{k+1}^i\big(a_k+\bar{a}_k+\sum_{j\in-i}(b_k^j+\bar{b}^j_k)\delta_k^j \big)(b_k^i+\bar{b}^i_k),\\
    G_k^i&:=(b_k^i+\bar{b}^i_k)\Big(\bar{\alpha}_{k+1}^i\big(\sum_{j\in-i}(b_k^j+\bar{b}^j_k) \bar{\delta}_k^j+c_k\big)+\beta_{k+1}^i\Big) -{\mu_k^{i\star}}^{\prime}\bar{n}_k^{ii},\\
    L_k^i&:=(q_k^i+\bar{q}_k^i)+\sum_{j\in -i}(r_k^{ij}+\bar{r}_k^{ij})(\delta_k^j)^2-\bar{\alpha}_k^i\\
    &+\bar{\alpha}_{k+1}^i\big(a_k+\bar{a}_k+\sum_{j\in-i}(b_k^j+\bar{b}^j_k)\delta_k^j\big)^2 -\tfrac{(F_k^i)^2}{D_k^i},\\
    M_k^i&:=\sum_{j\in -i}(r_k^{ij}+\bar{r}_k^{ij})\delta_k^j\bar{\delta}_k^j-{\mu_k^{i\star}}^{\prime}\big(\bar{m}_k^i+\sum_{j\in -i}\bar{n}_k^{ij}\delta_k^j\big)\\
    &+\big(a_k+\bar{a}_k+\sum_{j\in-i}(b_k^j+\bar{b}^j_k)\delta_k^j \big)\big(\bar{\alpha}_{k+1}^i\big(\sum_{j\in-i}(b_k^j+\bar{b}^j_k) \bar{\delta}_k^j
      \end{align*}
\begin{align*}
    &+c_k\big)+\beta_{k+1}^i\big)-\beta_{k}^i-\tfrac{F_k^iG_k^i}{D_k^i},~ \text{and}\\
    N_k^i&:=\tfrac{1}{2}\sum_{j\in-i}(r_k^{ij}+\bar{r}_k^{ij})(\bar{\delta}_k^j)^2-{\mu_k^{i\star}}^{\prime}\big(\sum_{j\in -i}\bar{n}_k^{ij}\bar{\delta}_k^j+p_k^i\big)\\
    &+\tfrac{1}{2}\alpha_{k+1}^i(\sigma_k)^2 \E[(w_k)^2]+\tfrac{1}{2}\bar{\alpha}_{k+1}^i\big(\sum_{j\in-i}(b_k^j+\bar{b}^j_k) \bar{\delta}_k^j+c_k\big)^2\\
    &+\beta_{k+1}^i\big(\sum_{j\in-i}(b_k^j+\bar{b}^j_k) \bar{\delta}_k^j+c_k\big)+\gamma_{k+1}^i-\gamma_k^i-\tfrac{1}{2}\tfrac{(G_k^i)^2}{D_k^i}.%
\end{align*}
\noindent  
 \textbf{Step 4 -- (Incorporation of inequality constraints and completion of squares):}
 We add and subtract the term $\sum_{k\in\K_l}{\mu_k^{i\star}}^{\prime}\big((\bar{m}_k^i+\sum_{j\in -i}\bar{n}_k^{ij}\delta_k^j)\E[x_k]+\bar{n}^{ii}_k \E[u_k^i]+\sum_{j\in -i}\bar{n}_k^{ij}\bar{\delta}_k^j+p_k^i\big)$ to the right-hand-side of the expression $\E[J^i(\bu^{i}, \bu^{-i})-f^i(0,x_0)]$ in \eqref{eq:PiLagrangian1}. Then, we perform the completion of squares of the terms involving $(u_k^i-\E[u_k^i])$ and $\E[u_k^i]$ as follows  
\begin{subequations}\label{eq:csquares}
\begin{align}
    &\tfrac{1}{2}A_k^i\E[(u_k^i-\E[u_k^i])^2]+B_k^i\E[(x_k-\E[x_k])(u_k^i-\E[u_k^i])]\nonumber\\
    &\quad =\tfrac{1}{2}A_k^i\E\Big[\big((u_k^i-\E[u_k^i])+\tfrac{B_k^i}{A_k^i}(x_k-\E[x_k])\big)^2\Big]\nonumber\\
    &\quad\quad-\tfrac{1}{2}\tfrac{(B^i_k)^2}{A_k^i}\E[(x_k-\E[x_k])^2],\\
    &\tfrac{1}{2}D_k^i\E[u_k^i]^2+\big(F_k^i\E[x_k]+G_k^i\big)\E[u_k^i]\nonumber\\
    &\quad=\tfrac{1}{2}D_k^i\big(\E[u_k^i]+\tfrac{1}{D_k^i}\big(F_k\E[x_k]+G_k^i\big)\big)^2-\tfrac{1}{2}\tfrac{(F^i_k)^2}{D_k^i}\E[x_k]^2\nonumber\\
    &\quad\quad-\tfrac{F_k^iG_k^i}{D_k^i}\E[x_k]-\tfrac{1}{2}\tfrac{(G^i_k)^2}{D_k^i}.
\end{align}    
\end{subequations}
After performing the  above calculations we obtain 
\begin{align}
    &\E[J^i(\bu^i, \bu^{-i\star})]=\E[f^i(0,x_0)]+\sum_{k\in\K_l}{\mu_k^{i\star}}^{\prime}\big((\bar{m}_k^i+\sum_{j\in -i}\bar{n}_k^{ij}\delta_k^j)\E[x_k]\nonumber\\
    &+\bar{n}^{ii}_k \E[u_k^i]+\sum_{j\in -i}\bar{n}_k^{ij}\bar{\delta}_k^j+p_k^i\big)+\tfrac{1}{2}(q_K^i-\alpha_K^i)\E[\big(x_K-\E[x_K]\big)^2]\nonumber\\
    &+\tfrac{1}{2}(q_K^i+\bar{q}_K^i-\bar{\alpha}_K^i)\E[x_K]^2-\beta_K^i\E[x_K]-\gamma_K^i\nonumber\\
    &+\tfrac{1}{2}\sum_{k\in\K_l}\big(C_k^i\E[(x_k-\E[x_k])^2]+L_k^i\E[x_k]^2+{2}M_k^iE[x_k]+{2}N_k^i\big)\nonumber\\
    &+\tfrac{1}{2}\sum_{k\in\K_l}A_k^i\E\big[\big((u_k^i-\E[u_k^i])+\tfrac{B_k^i}{A_k^i}(x_k-\E[x_k])\big)^2\big]\nonumber\\
    &+\tfrac{1}{2}\sum_{k\in\K_l}D_k^i\big(\E[u_k^i]+\tfrac{1}{D_k^i}\big(F_k^i\E[x_k]+G_k^i\big)\big)^2.\label{eq:PiLagrangian2}
\end{align}
\noindent  
\textbf{Step 5 -- (Verification of MFTGNE \eqref{eq:MFTNEdef}):}
Next, using the definitions of $A_k^i$, $B_k^i$, $D_k^i$, $F_k^i$ and $G_k^i$ in \eqref{eq:coefficients} it is easy to verify that $\eta_k^i=-B_k^i/A_k^i$, $\delta_k^i=-F_k^i/D_k^i$ and $\bar{\delta}_k^i=-G_k^{i}/D_k^i$.  Then, using $\eta_k^i=-B_k^i/A_k^i$ and  \eqref{eq:alpharecursive}, we obtain
\begin{align*}
    C_k^i&=q_k^i+\sum_{j\in -i}r_k^{ij}(\eta_k^j)^2+\alpha_{k+1}^i\big(a_k+\sum_{j\in-i}b^j_k\eta_k^j\big)^2-\alpha_k^i-\tfrac{(B_k^i)^2}{A_k^i}\\
    &=q_k^i+\sum_{j\in \N}r_k^{ij}(\eta_k^j)^2+\alpha_{k+1}^i\big(a_k+\sum_{j\in\N}b^j_k\eta_k^j\big)^2-\alpha_k^i-\tfrac{(B_k^i)^2}{A_k^i}\\
    &\quad-2\alpha_{k+1}^i\big(a_k+\sum_{j\in-i}b^j_k\eta_k^j\big)b_k^i\eta_k^i-(r_k^{ii}+\alpha_{k+1}^i(b_k^i)^2)(\eta_k^i)^2\\
    &=q_k^i+\sum_{j\in \N}r_k^{ij}(\eta_k^j)^2+\alpha_{k+1}^i(\mathsf{A}_k)^2\\
    &\quad-\alpha_k^i-\tfrac{(B_k^i)^2}{A_k^i}+2\tfrac{(B_k^i)^2}{A_k^i}-\tfrac{(B_k^i)^2}{A_k^i}=0.
\end{align*}
Similarly using $\delta_k^i=-F_k^i/D_k^i$ and \eqref{eq:baralpharecursive}, we observe that
\begin{align*}
    L_k^i:&=(q_k^i+\bar{q}_k^i)+\sum_{j\in -i}(r_k^{ij}+\bar{r}_k^{ij})(\delta_k^j)^2-\bar{\alpha}_k^i\\
    &\quad+\bar{\alpha}_{k+1}^i\big(a_k+\bar{a}_k+\sum_{j\in-i}(b_k^j+\bar{b}^j_k)\delta_k^j\big)^2 -\tfrac{(F_k^i)^2}{D_k^i}\\
    &=(q_k^i+\bar{q}_k^i)+\sum_{j\in \N}(r_k^{ij}+\bar{r}_k^{ij})(\delta_k^j)^2\\
    &\quad+\bar{\alpha}_{k+1}^i\big(a_k+\bar{a}_k+\sum_{j\in\N}(b_k^j+\bar{b}^j_k)\delta_k^j\big)^2-\bar{\alpha}_k^i-\tfrac{(F_k^i)^2}{D_k^i}\\
    &\quad-2\bar{\alpha}_{k+1}^i\big(a_k+\bar{a}_k+\sum_{j\in -i}(b_k^j+\bar{b}^j_k)\delta_k^j\big)(b_k^i+\bar{b}^i_k)\delta_k^i\\
    &\quad-\big(r_k^{ii}+\bar{r}_k^{ii}+\bar{\alpha}_{k+1}^i(b_k^i+\bar{b}^i_k)^2\big)(\delta_k^j)^2\\
    &=(q_k^i+\bar{q}_k^i)+\sum_{j\in \N}(r_k^{ij}+\bar{r}_k^{ij})(\delta_k^j)^2+\bar{\alpha}_{k+1}^i(\bar{\mathsf{A}}_k)^2-\bar{\alpha}_k^i\\
    &\quad-\tfrac{(F_k^i)^2}{D_k^i}+2\tfrac{(F_k^i)^2}{D_k^i}-\tfrac{(F_k^i)^2}{D_k^i}=0.
\end{align*}
Also using $\delta_k^i=-F_k^i/D_k^i$, $\bar{\delta}_k^i=-G_k^{i}/D_k^i$, \eqref{eq:Betarecursive}, \eqref{eq:gammarecursive} and the definition of $G_k^i$, we simplify the expressions for $M_k^i$ and $N_k^i$ as follows
\begin{align*}
    M_k^i&:=\sum_{j\in -i}(r_k^{ij}+\bar{r}_k^{ij})\delta_k^j\bar{\delta}_k^j-{\mu_k^{i\star}}^{\prime}\big(\bar{m}_k^i+\sum_{j\in -i}\bar{n}_k^{ij}\delta_k^j\big)\\
    &\quad+\big(a_k+\bar{a}_k+\sum_{j\in-i}(b_k^j+\bar{b}^j_k)\delta_k^j \big)\big(\bar{\alpha}_{k+1}^i\big(\sum_{j\in-i}(b_k^j+\bar{b}^j_k) \bar{\delta}_k^j\\
    &\quad+c_k\big)+\beta_{k+1}^i\big)-\beta_{k}^i-\tfrac{F_k^iG_k^i}{D_k^i}\\
    &=\sum_{j\in -i}(r_k^{ij}+\bar{r}_k^{ij})\delta_k^j\bar{\delta}_k^j-{\mu_k^{i\star}}^{\prime}\big(\bar{m}_k^i+\sum_{j\in\N}\bar{n}_k^{ij}\delta_k^j\big)\\
    &\quad+\big(a_k+\bar{a}_k+\sum_{j\in\N}(b_k^j+\bar{b}^j_k)\delta_k^j \big)\big(\bar{\alpha}_{k+1}^i\big(\sum_{j\in-i}(b_k^j+\bar{b}^j_k) \bar{\delta}_k^j\\
    &\quad+c_k\big)+\beta_{k+1}^i\big)-\beta_{k}^i+{\mu_k^{i\star}}^{\prime}\bar{n}_k^{ii}\delta_k^i+\delta_k^iG_k^i,~(\textrm{as}~\delta_k^i=-\tfrac{F_k^i}{D_k^i})\\
    &\quad-(b_k^i+\bar{b}^i_k)\big(\bar{\alpha}_{k+1}^i\big(\sum_{j\in-i}(b_k^j+\bar{b}^j_k) \bar{\delta}_k^j+c_k\big)+\beta_{k+1}^i\big)\delta_k^i\\
    &=\sum_{j\in -i}(r_k^{ij}+\bar{r}_k^{ij})\delta_k^j\bar{\delta}_k^j-{\mu_k^{i\star}}^{\prime}\big(\bar{m}_k^i+\sum_{j\in\N}\bar{n}_k^{ij}\delta_k^j\big)\\
    &\quad+\bar{\mathsf{A}}_k\Big(\bar{\alpha}_{k+1}^i\big(\sum_{j\in-i}(b_k^j+\bar{b}^j_k) \bar{\delta}_k^j+c_k\big)+\beta_{k+1}^i\Big)-\beta_{k}^i\\
    &\quad+\delta_k^i(G_k^i+{\mu_k^{i\star}}^{\prime}\bar{n}_k^{ii})-\delta_k^i(G_k^i+{\mu_k^{i\star}}^{\prime}\bar{n}_k^{ii})=0,
\end{align*}
\begin{align*}
    N_k^i&:=\tfrac{1}{2}\sum_{j\in-i}(r_k^{ij}+\bar{r}_k^{ij})(\bar{\delta}_k^j)^2-{\mu_k^{i\star}}^{\prime}\big(\sum_{j\in -i}\bar{n}_k^{ij}\bar{\delta}_k^j+p_k^i\big)\\
    &\quad+\tfrac{1}{2}\alpha_{k+1}^i(\sigma_k)^2 \E[(w_k)^2]+\tfrac{1}{2}\bar{\alpha}_{k+1}^i\big(\sum_{j\in-i}(b_k^j+\bar{b}^j_k) \bar{\delta}_k^j+c_k\big)^2\\
    &\quad+\beta_{k+1}^i\big(\sum_{j\in-i}(b_k^j+\bar{b}^j_k) \bar{\delta}_k^j+c_k\big)+\gamma_{k+1}^i-\gamma_k^i-\tfrac{1}{2}\tfrac{(G_k^i)^2}{D_k^i}\\
    &=\tfrac{1}{2}\sum_{j\in-i}(r_k^{ij}+\bar{r}_k^{ij})(\bar{\delta}_k^j)^2-{\mu_k^{i\star}}^{\prime}\big(\sum_{j\in -i}\bar{n}_k^{ij}\bar{\delta}_k^j+p_k^i\big)\\
    &\quad+\tfrac{1}{2}\alpha_{k+1}^i(\sigma_k)^2 \E[(w_k)^2]+\tfrac{1}{2}\big(\sum_{j\in-i}(b_k^j+\bar{b}^j_k) \bar{\delta}_k^j+c_k\big)\\
    &\quad \times\big(\bar{\alpha}_{k+1}^i(\sum_{j\in-i}(b_k^j+\bar{b}^j_k) \bar{\delta}_k^j+c_k)+\beta_{k+1}^i\big)\\
    &\quad+\tfrac{1}{2}\beta_{k+1}^i\big(\sum_{j\in-i}(b_k^j+\bar{b}^j_k) \bar{\delta}_k^j+c_k\big)+\gamma_{k+1}^i-\gamma_k^i+\tfrac{1}{2}G_k^i\bar{\delta}_k^i
\end{align*}
\begin{align*}
    &=\tfrac{1}{2}\sum_{j\in-i}(r_k^{ij}+\bar{r}_k^{ij})(\bar{\delta}_k^j)^2-{\mu_k^{i\star}}^{\prime}\big(\sum_{j\in -i}\bar{n}_k^{ij}\bar{\delta}_k^j+p_k^i\big)\\
    &\quad+\tfrac{1}{2}\alpha_{k+1}^i(\sigma_k)^2 \E[(w_k)^2]+\tfrac{1}{2}\big(\sum_{j\in\N}(b_k^j+\bar{b}^j_k) \bar{\delta}_k^j+c_k\big)\\
    &\quad \times\big(\bar{\alpha}_{k+1}^i(\sum_{j\in-i}(b_k^j+\bar{b}^j_k) \bar{\delta}_k^j+c_k)+\beta_{k+1}^i\big)\\
    &\quad+\tfrac{1}{2}\beta_{k+1}^i\big(\sum_{j\in-i}(b_k^j+\bar{b}^j_k) \bar{\delta}_k^j+c_k\big)+\gamma_{k+1}^i-\gamma_k^i\\
    &\quad-\tfrac{1}{2}(G_k^i+{\mu_k^{i\star}}^{\prime}\bar{n}_k^{ii})\bar{\delta}_k^i+\tfrac{1}{2}G_k^i\bar{\delta}_k^i=0.
\end{align*}
As the terms $C_k^i$, $L_k^i$, $M_k^i$ and $N_k^i$ are zero for $k\in \K_l$ and $\alpha_K^i=q_K^i$, $\bar{\alpha}_K^i=q_K^i+\bar{q}_K^i$, $\beta_K^i=0$,  $\gamma_K^i=0$, the expression  \eqref{eq:PiLagrangian2} is simplified as follows 
\begin{align}
    &\E[J^i(\bu^i, \bu^{-i\star})]=\E[f^i(0,x_0)]+\sum_{k\in\K_l}{\mu_k^{i\star}}^{\prime}\big((\bar{m}_k^i+\sum_{j\in -i}\bar{n}_k^{ij}\delta_k^j)\nonumber\\
    &\quad\times\E[x_k]+\bar{n}^{ii}_k \E[u_k^i]+\sum_{j\in -i}\bar{n}_k^{ij}\bar{\delta}_k^j+p_k^i\big)\nonumber \\
    &\quad+\tfrac{1}{2}\sum_{k\in\K_l}A_k^i\E\big[\big((u_k^i-\E[u_k^i])-\eta_k^i(\bar{x}_k-\E[x_k])\big)^2\big]\nonumber\\
    &\quad+\tfrac{1}{2}\sum_{k\in\K_l}D_k^i\big(\E[u_k^i]-\delta_k^i\E[x_k]-\bar{\delta}_k^i\big)^2.\label{eq:StaticObj1}
\end{align}
If $\bu^i$ is set as $\bu^{i\star}$ (given by \eqref{eq:MFT_NE}) in \eqref{eq:StaticObj1}, then we know $\{x^\star_k,~k\in \K\}$ is the corresponding state trajectory which satisfies the complementarity condition \eqref{eq:KKTcomp} and evolves according to \eqref{eq:Eqstate}. 
Then, the second term on the right-hand-side of the expression in \eqref{eq:StaticObj1} vanishes. Besides this, the third and the fourth terms also vanish as $\bu^{i\star}$ satisfies \eqref{eq:MFT_NE}. So, we have
\begin{align}
    &\E[J^i(\bu^{i\star}, \bu^{-i\star})]=\E[f^i(0,x_0)].\label{eq:StaticObj2}
\end{align}
From Assumption \ref{ass:GameAssumption}.\ref{item:1}, every 
 admissible $\bu^i\in \U^i(\bu^{-i\star})$ satisfies $(\bar{m}_k^i+\sum_{j\in -i}\bar{n}_k^{ij}\delta_k^j)\E[x_k]+\bar{n}^{ii}_k \E[u_k^i]+\sum_{j\in -i}\bar{n}_k^{ij}\bar{\delta}_k^j+p_k^i\geq 0$.
Further, the multipliers in \eqref{eq:KKTcomp} satisfy $\mu_k^{i\star} {\geq 0}$ for all $k\in \K_l$.
This implies, the second term on right-hand-side of the expression in \eqref{eq:StaticObj1}
is non-negative. Besides this, as $A_k^i,D_k^i >0$ for all $k\in\K_l$ the third and the fourth terms are also non-negative. Consequently, comparing \eqref{eq:StaticObj1} and \eqref{eq:StaticObj2}, we obtain
 \begin{align*}
     \E[J_i( \bu^{i\star}, \bu^{-i\star})] \leq \E[J_i(\bu^i, \bu^{-i\star})],~\forall \bu^i\in \U^i(\bu^{-i\star}).
\end{align*}
As the choice of player $i$ is arbitrary, the above condition holds for each player $i\in\N$. So, from Definition \ref{def:MFTNEdef}, the strategy profile $\{u_k^{i\star},~i\in\N,~k\in\K_l\}$ given by \eqref{eq:MFT_NE} is indeed a MFTGNE. 
\end{proof}
\begin{remark} \label{rem:parameterdependency}
Using  the algebraic equations \eqref{eq:coefficients}, at each stage $k\in \K_l$, we note that the variables $\{\eta_k^i,\delta_k^i,\bar{\delta}_k^i, i\in\N\}$  can be solved interms of the variables $\{\alpha_{k+1}^i,\bar{\alpha}_{k+1}^i,\beta_{k+1}^i, i\in\N\}$. Then, using these solutions in the  backward difference equations \eqref{eq:BackwardRecursive} the variables $\{\alpha_{k}^i,\bar{\alpha}_{k}^i,\beta_{k}^i, i\in\N\}$ are evaluated. So, starting with the boundary conditions $\alpha_K^i=q_K^i$, $\bar{\alpha}_K^i=q_K^i+\bar{q}_K^i$ and $\beta_K^i=0$, and using the above mentioned recursive procedure the variables
	$\{\alpha_{k}^i,\bar{\alpha}_{k}^i,\beta_{k}^i,  \eta_k^i ,  \delta_k^i ,\bar{\delta}_k^i,~k\in \K_l,~i\in \N \}$ are determined. In particular, from \eqref{eq:bdeltaeq} and \eqref{eq:Betarecursive}, the  variables $\{\bar{\delta}_k^i,\beta_k^i,k\in \K_l,i\in \N\}$ contain linear terms involving the multipliers $\{\mu_k^{i\star},~k\in \K,~i\in \N\}$. Further, substituting $\{\bar{\delta}_k^i,~k\in \K_l,~i\in \N\}$ in \eqref{eq:Eqstate}, the MFTGNE state trajectory $\{x_k^\star,~k\in \K\}$  is expressed linearly in terms of the multipliers  $\{\mu_k^{i\star},~k\in \K,~i\in \N\}$. Upon eliminating these state variables in \eqref{eq:KKTcomp} we obtain (implicit) complementarity conditions involving only the multiplies
	and $\E[x_0]$ (which is known). In Section \ref{sec:Solvability}, under a few assumptions on the problem data, we illustrate the above mentioned procedure towards determining the multipliers.
\end{remark} 
\begin{remark} \label{rem:relaxedCond}
Assumption \ref{ass:GameAssumption}.\ref{item:3} can be relaxed with a less  stringent condition by requiring that the solutions  $\{\alpha_k^i,\bar{\alpha}_k^i,~i\in \N,~k\in \K\}$ of  the backward difference equations \eqref{eq:alpharecursive}-\eqref{eq:baralpharecursive} are such that    $A_k^i=r_k^{ii}+\alpha_{k+1}^i(b^i_k)^2$ and  $D_k^i=r_k^{ii}+\bar{r}_k^{ii}+\bar{\alpha}_{k+1}^i(b_k^i+\bar{b}^i_k)^2$  are positive for all $k\in \K_l$ and $i\in \N$. We notice that $\{\alpha_k^i,\bar{\alpha}_k^i,~i\in \N,~k\in \K\}$ depend only on the problem data associated with state dynamics \eqref{eq:MFstate} and objectives \eqref{eq:objective1}.
	So, using the recursive procedure mentioned in Remark \ref{rem:parameterdependency},  the required positivity condition can be verified numerically using the problem data; see also Remark \ref{rem:matrixinv}.
\end{remark}
\begin{remark}From Remark \ref{rem:relaxedCond} and \eqref{eq:StaticObj1}, we have that player $i$'s
		expected cost function is a strictly convex function in her decision variables $\bu^i$.
\end{remark}
\section{Solvability}\label{sec:Solvability}  
In this section, we present an approach for reformulating the equations \eqref{eq:KKTcomp}-\eqref{eq:BackwardRecursive} as single large-scale linear complementarity problem.
This procedure is based on \cite{Reddy:15}, and involves elimination of state variables in the
equations \eqref{eq:KKTcomp}-\eqref{eq:Eqstate}; see also Remarks \ref{rem:parameterdependency} and
\ref{rem:relaxedCond}. We define all the notations used in this section as follows: $\mathbf{R}_k:=\oplus_{i=1}^Nr_k^{ii}$, $\bar{\mathbf{R}}_k:=\oplus_{i=1}^N(r_k^{ii}+\bar{r}_k^{ii})$, $\mathbf{B}_k:=\oplus_{i=1}^Nb_k^{i}$, $\bar{\mathbf{B}}_k:=\oplus_{i=1}^N(b_k^i+\bar{b}_k^i)$, $\mathsf{B}_k:=\row{b_k^i}_{i=1}^{N}$, 
$\mathsf{N}_k:=\oplus_{i=1}^N\bar{n}_k^{ii}$, 
$\bar{\mathsf{M}}_k:=\col{\bar{m}_k^i+\sum_{j\in\N}\bar{n}_k^{ij}\delta_k^j}_{i=1}^{N}$,
$\bm{\alpha}_{k}:=\col{\alpha_{k}^i}_{i=1}^N$, $\bar{\bm{\alpha}}_{k}:=\col{\bar{\alpha}_{k}^i}_{i=1}^N$, $\bm{\eta}_k:=\col{\eta_k^i}_{i=1}^{N}$, $\bm{\beta}_{k}:=\col{\beta_{k}^i}_{i=1}^N$, $\bm{\delta}_k:=\col{\delta_k^i}_{i=1}^{N}$, $\bar{\bm{\delta}}_k:=\col{\bar{\delta}_k^i}_{i=1}^{N}$, $\bm{\mu}_k^{\star}:=\col{\mu_k^{i\star}}_{i=1}^{N}$, 
$\Lambda_k:=\mathbf{R}_k+\mathbf{B}_k^{\prime}\bm{\alpha}_{k+1}\mathsf{B}_k$, $\bar{\Lambda}_k:=\bar{\mathbf{R}}_k+\bar{\mathbf{B}}_k^{\prime}\bar{\bm{\alpha}}_{k+1}\bar{\mathsf{B}}_k$, $[\mathsf{P}_{k}^1]_{ij}=(r_k^{ij}+\bar{r}_k^{ij})\delta_k^j+\bar{\mathsf{A}}_k\bar{\alpha}_{k+1}^i(b_k^j+\bar{b}^j_k)$ for $i\neq j$, $[\mathsf{P}_{k}^1]_{ii}=0$, $\mathsf{P}_{k}^2=\oplus_{i=1}^N(\bar{m}_k^i+\sum_{j\in\N}\bar{n}_k^{ij}\delta_k^j)^{\prime}$, $\mathsf{P}_{k}^3=\col{\bar{\mathsf{A}}_k\bar{\alpha}_{k+1}^i}_{i=1}^N$. 
Next, for all $k\in\K_l$, we aggregate the variables as $x_{\K}^{\star}=\col{x_{k}^{\star}}_{k=0}^{K-1}$, $\E[x_{\K}^{\star}]=\col{\E[x_{k}^{\star}]}_{k=0}^{K-1}$,  $\bu^{\star}_\K=\col{\bu^{\star}_k}_{k=0}^{K-1}$,  $\E[\bu_{\K}^{\star}]=\col{\E[\bu_{k}^{\star}]}_{k=0}^{K-1}$, $\bm{\mu}^{\star}_\K=\col{\bm{\mu}^{\star}_k}_{k=0}^{K-1}$, $\bm{c}_\K=\col{c_{k}}_{k=0}^{K-1}$, $\bm{\bar{\delta}}_\K=\col{\bm\bar{\delta}_{k}}_{k=0}^{K-1}$, $\bm{p}_\K=\col{\bm{p}_k}_{k=0}^{K-1}$, $\bm{w}_\K=\col{w_k}_{k=0}^{K-1}$, $\bm{\eta}_\K=\oplus_{k=0}^{K-1}\bm{\eta}_k$, 
$\bm{\delta}_\K=\oplus_{k=0}^{K-1}\bm{\delta}_k$, $\bar{\mathsf{M}}_\K=\oplus_{k=0}^{K-1}\bar{\mathsf{M}}_k$, $\bar{\mathsf{N}}_\K=\oplus_{k=0}^{K-1}\bar{\mathsf{N}}_k$. 
 Let   $\psi(k, \tau)$ and $\phi(k, \tau)$ be the state transition matrices associated with the system dynamics \eqref{eq:equistatediff} and \eqref{eq:equistate}, respectively i.e., $\psi(k, \tau)=\mathsf{A}_{k-1}\mathsf{A}_{k-2}\cdots\mathsf{A}_{\tau}$ for any $k > \tau$ and $\psi(k, \tau)=\mathbf{I}$ for $k=\tau$, $\phi(k, \tau)=\bar{\mathsf{A}}_{k-1}\bar{\mathsf{A}}_{k-2}\cdots\bar{\mathsf{A}}_{\tau}$ for any $k > \tau$ and $\phi(k, \tau)=\mathbf{I}$ for $k=\tau$. Using these, we define  $[\mathsf{P}_\K^1]_{k\tau}=\bar{\mathbf{B}}_{k+1}^{\prime}\big(\textbf{I}_N\otimes\phi(\tau-1,k)\big)\mathsf{P}_{\tau-1}^1$,  $[\mathsf{P}_\K^2]_{k\tau}=-\bar{\mathbf{B}}_{k+1}^{\prime}\big(\textbf{I}_N\otimes\phi(\tau-1,k)\big)\mathsf{P}_{\tau-1}^2$,  $[\mathsf{P}_\K^3]_{k\tau}=-\bar{\mathbf{B}}_{k+1}^{\prime}\big(\textbf{I}_N\otimes\phi(\tau-1,k)\big)\mathsf{P}_{\tau-1}^3$ for $\tau>k$, $[\mathsf{P}_\K^1]_{kk}=\bar{\Lambda}_k$, $[\mathsf{P}_\K^2]_{kk}=\mathsf{N}_k^{\prime}$, $[\mathsf{P}_\K^3]_{kk}=-\bar{\mathbf{B}}_k^{\prime}\bar{\bm{\alpha}}_{k+1}$ and $[\mathsf{P}_\K^1]_{k\tau}=\mathbf{0}$, $[\mathsf{P}_\K^2]_{k\tau}=\mathbf{0}$, $[\mathsf{P}_\K^3]_{k\tau}=\mathbf{0}$ for $\tau <k$, $[\bm{\Psi}_0]_{k}=\psi(k-1, 0)$, $[\bm{\Psi}_1]_{k\tau}=\psi(k-1, \tau)\sigma_{\tau-1}$, for $k>\tau$, $[\bm{\Psi}_1]_{k\tau}=\mathbf{0}$ for $k\leq \tau$,  $[\bm{\Phi}_0]_{k}=\phi(k-1, 0)$, $[\bm{\Phi}_1]_{k \tau}=\phi(k-1, \tau)\bar{\mathsf{B}}_{\tau-1}$, $[\bm{\Phi}_2]_{k \tau}=\phi(k-1, \tau)$ for $k >\tau$ and $[\bm{\Phi}_1]_{k \tau}=\mathbf{0}$,  $[\bm{\Phi}_2]_{k \tau}=\mathbf{0}$ for $k \leq \tau$ with  $k,\tau \in\K_r$.

  Using the above notations for all $i\in\N$,   \eqref{eq:MFT_NE} can be written compactly for all $k\in\K_l$ as
\begin{subequations}\label{eq:MFT_NEv}
    \begin{align}
        \bu_\K^{\star}-\E[\bu_\K^{\star}]&=\bm{\eta}_\K(x_\K^{\star}-\E[x_\K^{\star}]),\label{eq:Unoise}\\
    \E[\bu_\K^{\star}]&=\bm{\delta}_\K\E[x_\K^{\star}]+\bm{\bar{\delta}}_\K.\label{eq:uforallK}
    \end{align}
\end{subequations}
Similarly, for all $i\in\N$, \eqref{eq:coefficients} are given by
\begin{subequations}\label{eq:lambdavect}
    \begin{align}
        \Lambda_k\bm{\eta}_k&=-\mathbf{B}_k^{\prime}\bm{\alpha}_{k+1}a_k,\\
        \bar{\Lambda}_k\bm{\delta}_k&=-\bar{\mathbf{B}}_k^{\prime}\bar{\bm{\alpha}}_{k+1}(a_k+\bar{a}_k),\\
        \bar{\Lambda}_k\bar{\bm{\delta}}_k&=-\bar{\mathbf{B}}_k^{\prime}\bar{\bm{\alpha}}_{k+1}c_k-\bar{\mathbf{B}}_k^{\prime}\bm{\beta}_{k+1}+\mathsf{N}_k^{\prime}\bm{\mu}_k^{\star}\label{eq:VBarDeltarecursive1}.
    \end{align}
\end{subequations}
Also the vector form representation of \eqref{eq:Betarecursive} is given as
\begin{align}
    \bm{\beta}_k&=(\textbf{I}_N\otimes\bar{\mathsf{A}}_k)\bm{\beta}_{k+1}+\mathsf{P}_{k}^1\bm{\bar{\delta}}_{k}+\mathsf{P}_{k}^2\bm{\mu}_{k}^{\star}+\mathsf{P}_{k}^3c_{k}\nonumber\\
    &=\sum_{\tau=k}^{K-1}\big(\textbf{I}_N \otimes \phi(\tau, k)\big)\big(\mathsf{P}_{\tau}^1\bm{\bar{\delta}}_{\tau}+\mathsf{P}_{\tau}^2\bm{\mu}_{\tau}^{\star}+\mathsf{P}_{\tau}^3c_{\tau}\big),\label{eq:VBetarecursive}
\end{align}
along with boundary condition $\bm{\beta}_{K}=\mathbf{0}$. Next using \eqref{eq:VBetarecursive} in \eqref{eq:VBarDeltarecursive1}, we obtain
\begin{align*}
     \bar{\Lambda}_k\bar{\bm{\delta}}_k&=-\bar{\mathbf{B}}_k^{\prime}\sum_{\tau=k+1}^{K-1}\big(\textbf{I}_N\otimes\phi(\tau, k+1)\big)\mathsf{P}_{\tau}^1\bm{\bar{\delta}}_{\tau}\\
     &\quad-\bar{\mathbf{B}}_k^{\prime}\sum_{\tau=k+1}^{K-1}\big(\textbf{I}_N\otimes\phi(\tau, k+1)\big)\mathsf{P}_{\tau}^2\bm{\mu}_{\tau}^{\star}+\mathsf{N}_k\bm{\mu}_k^{\star}\\
     &\quad-\bar{\mathbf{B}}_k^{\prime}\bar{\bm{\alpha}}_{k+1}c_k-\bar{\mathbf{B}}_k^{\prime}\sum_{\tau=k+1}^{K-1}\big(\textbf{I}_N\otimes\phi(\tau, k+1)\big)\mathsf{P}_{\tau}^3c_{\tau}.
\end{align*}
 Collecting all the terms for $k\in \K_l$ and using the definitions of $\mathsf{P}_\K^1$, $\mathsf{P}_\K^2$ and $\mathsf{P}_\K^3$, we obtain
\begin{align}
    \mathsf{P}_\K^1\bar{\bm{\delta}}_\K=\mathsf{P}_\K^2\bm{\mu}_{\K}^{\star}+\mathsf{P}_\K^3\bm{c}_{\K}.\label{eq:VBarDeltarecursive2}
\end{align}
We have the following assumption.
\begin{assumption}\label{ass:Matrixinvertivility}
    The matrices $\{\Lambda_k,~\bar{\Lambda}_k,~ k\in\K_l\}$ are invertible.
\end{assumption}
 %
  In \eqref{eq:VBarDeltarecursive2}, $\mathsf{P}_\K^1$ is a upper triangular matrix (as $[\mathsf{P}_\K^1]_{k\tau}=\mathbf{0}$ for $\tau <k$) with $\bar{\Lambda}_k$, $k\in\K_l$ as the block diagonal elements.  From Assumption \ref{ass:Matrixinvertivility}, the matrix $\mathsf{P}_\K^1$ is also invertible. Then, from    \eqref{eq:VBarDeltarecursive2} we have
\begin{align}
    \bar{\bm{\delta}}_\K=(\mathsf{P}_\K^1)^{-1}\mathsf{P}_\K^2\bm{\mu}_{\K}^{\star}+(\mathsf{P}_\K^1)^{-1}\mathsf{P}_\K^3\bm{c}_{\K}.\label{eq:DeltaforallK}
\end{align}
The equilibrium state trajectory \eqref{eq:Eqstate}  is solved as follows
\begin{align*}
    &x_{k}^{\star}-\E[x_{k}^{\star}]=\mathsf{A}_{k-1}(x_{k-1}^{\star}-\E[x_{k-1}^{\star}])+\sigma_{k-1} w_{k-1}\\
    &=\psi(k, 0)(x_0-\E[x_0])+\sum_{\tau=1}^{k}\psi(k, \tau)\sigma_{\tau-1} w_{\tau-1},\\
    &\E[x_{k}^{\star}]= \bar{\mathsf{A}}_{k-1}\E[x_{k-1}^{\star}]+\bar{\mathsf{B}}_{k-1}\bar{\bm{\delta}}_{k-1}+c_{k-1}\nonumber\\
    &=\phi(k, 0)\E[x_{0}]+\sum_{\tau=1}^{k}\phi(k, \tau)\bar{\mathsf{B}}_{\tau-1}\bar{\bm{\delta}}_{\tau-1}+\sum_{\tau=1}^{k}\phi(k, \tau)c_{\tau-1}.
\end{align*}
We write the above equations for all $k\in\K_l$  compactly as follows
\begin{subequations}
\begin{align}
    &x_\K^{\star}-\E[x_\K^{\star}] = \bm{\Psi}_0(x_0-\E[x_0])+\bm{\Psi}_1\bm{w}_\K,\label{eq:noise}\\
    &\E[x_\K^{\star}] = \bm{\Phi}_0\E[x_0]+\bm{\Phi}_1\bar{\bm{\delta}}_\K+\bm{\Phi}_2\bm{c}_\K .\label{eq:Finalstate}
\end{align}
Also, for all $i\in\N$ the complementarity condition \eqref{eq:KKTcomp} is given by
\begin{align*}
    0\leq  &\bar{\mathsf{M}}_k\E[x_k^{\star}]+\bar{\mathsf{N}}_k\bar{\bm{\delta}}_k+\bm{p}_k \perp \bm{\mu}_k^{\star} \geq 0,
\end{align*}
 and compactly for all $k\in\K_l$ represented as
\begin{align}
    &0\leq \bar{\mathsf{M}}_\K\E[x_\K^{\star}]+\bar{\mathsf{N}}_\K\bar{\bm{\delta}}_\K+\bm{p}_\K \perp  \bm{\mu}_\K^{\star} \geq 0.\label{eq:vCC}
\end{align}
\end{subequations}
\begin{theorem}\label{th:MFTNElcp}
	Let  Assumptions \ref{ass:GameAssumption} and \ref{ass:Matrixinvertivility} hold. Then, the  MFTGNE strategy profile for MFTDG is given by 
	\begin{subequations}\label{eq:FinalMFT_NE}
    \begin{align}
        \bu_\K^{\star}-\E[\bu_\K^{\star}]&=\bm{\eta}_\K\bm{\Psi}_0(x_0-\E[x_0])+\bm{\eta}_\K\bm{\Psi}_1\bm{w}_\K,\label{eq:FinalUnoise}\\
    \E[\bu_\K^{\star}]&=\mathsf{F}\bm{\mu}_{\K}^{\star}+\mathsf{P},\label{eq:FinaluNE}
    \end{align}
\end{subequations}
	with $\bm{\mu}_{\K}^{\star}$ being the solution of the following single large-scale linear complementarity problem  
	\begin{align}
	\texttt{LCP}:\quad	0 \leq \mathsf{M}\bm{\mu}_{\K}^{\star}+\mathsf{Q} \perp \bm{\mu}_{\K}^{\star} \geq 0, \label{eq:FinalLCP}
	\end{align}
	where $\mathsf{F}=(\bm{\delta}_\K\bm{\Phi}_1+\mathbf{I})(\mathsf{P}_\K^1)^{-1}\mathsf{P}_\K^2$,  $\mathsf{P}=\bm{\delta}_\K\bm{\Phi}_0\E[x_0]+\big((\bm{\delta}_\K\bm{\Phi}_1+\mathbf{I})(\mathsf{P}_\K^1)^{-1}\mathsf{P}_\K^3+\bm{\delta}_\K\bm{\Phi}_2\big)\bm{c}_{\K}$,
    $\mathsf{M}=(\bar{\mathsf{M}}_\K\bm{\Phi}_1+\bar{\mathsf{N}}_\K)(\mathsf{P}_\K^1)^{-1}\mathsf{P}_\K^2$, $\mathsf{Q}=\bar{\mathsf{M}}_\K\bm{\Phi}_0\E[x_0]+\big((\bar{\mathsf{M}}_\K\bm{\Phi}_1+\bar{\mathsf{N}}_\K)(\mathsf{P}_\K^1)^{-1}\mathsf{P}_\K^3+\bar{\mathsf{M}}_\K\bm{\Phi}_2\big)\bm{c}_\K+\bm{p}_\K$.
\end{theorem}
\begin{proof}
      Substituting \eqref{eq:noise} in \eqref{eq:Unoise} results in \eqref{eq:FinalUnoise}. Using \eqref{eq:Finalstate} in \eqref{eq:uforallK} and \eqref{eq:vCC}  we get
     \begin{align*}
         &\E[\bu_\K^{\star}]=(\bm{\delta}_\K\bm{\Phi}_1+\mathbf{I})\bm{\bar{\delta}}_\K+\bm{\delta}_\K(\bm{\Phi}_0\E[x_0]+\bm{\Phi}_2\bm{c}_\K),\\
         &0\leq (\bar{\mathsf{M}}_\K\bm{\Phi}_1+\bar{\mathsf{N}}_\K)\bar{\bm{\delta}}_\K+\bar{\mathsf{M}}_\K\bm{\Phi}_0\E[x_0]+\bar{\mathsf{M}}_\K\bm{\Phi}_2\bm{c}_\K+\bm{p}_\K \perp  \bm{\mu}_\K^{\star} \geq 0.
     \end{align*}
     Finally, using the expression for $\bm{\bar{\delta}}_\K$  from \eqref{eq:DeltaforallK} in the above equations we obtain \eqref{eq:FinaluNE} and \eqref{eq:FinalLCP}, respectively. 
\end{proof}
 \begin{remark}\label{rem:matrixinv}
   We note that equations \eqref{eq:lambdavect} are  a matrix representation of the algebraic equations \eqref{eq:coefficients} at stage $k\in \K_l$, with $\Lambda_k=\mathbf{R}_k+\mathbf{B}_k^{\prime}\bm{\alpha}_{k+1}\mathsf{B}_k$ and $\bar{\Lambda}_k=\bar{\mathbf{R}}_k+\bar{\mathbf{B}}_k^{\prime}\bar{\bm{\alpha}}_{k+1}\bar{\mathsf{B}}_k$. Following the recursive procedure outlined in Remarks \ref{rem:parameterdependency} and \ref{rem:relaxedCond}, the invertibilty of these matrices at every stage $k\in \K_l$, as required by Assumption \ref{ass:Matrixinvertivility}, 
 	can be verified using the problem data without the need for solving the \texttt{LCP}. 
\end{remark}
\begin{remark}\label{rem:multipleNash} We note that the \texttt{LCP} given by \eqref{eq:FinalLCP} is an implicit representation of \eqref{eq:KKTcomp}-\eqref{eq:Eqstate}. Further, if the 	\texttt{LCP} has multiple solutions, then, from  Theorem \ref{thm:MFT_NE},  each one of these solutions   constitutes a MFTGNE. The existence conditions and numerical methods for the \texttt{LCP}  have been extensively studied in the optimization community; refer to \cite{Pang1:09} for  details.
\end{remark}
%
 

\section{Numerical Illustration}\label{sec:Numerical}
\begin{figure}[h]
	\centering
	\includegraphics[scale=0.8]{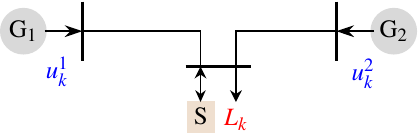}
	\caption{A microgrid   with two generators and one storage unit}
	\label{fig:Fig1}
\end{figure} 
In recent years, game theory has been extensively used to analyze energy storage issues that arise in microgrid management; see \cite{Chen:16} and \cite{Gomez:19}. Motivated by these studies, we consider a simplified microgrid model, as shown in Fig. \ref{fig:Fig1}. The model comprises two generators, $\mathrm{G_1}$ and $\mathrm{G_2}$, with generation levels $u_k^1$ and $u_k^2$, respectively. These generators supply power to a time-varying load  represented by $L_k$, through transmission lines. A storage unit, $\mathrm{S}$, is installed near the load, which can either store surplus generator output (through charging) or supply the load (through discharging) when demand is not met.
Let $x_k$ denote the storage level at time instant $k$, and its evolution due to charging and discharging be given by $
    x_{k+1}=a x_k+\sum_{i=1}^{2}b ^iu_k^i-L_k+\sigma  w_k$,
where $a \in (0,1)$ and $b^1, b^2 \in (0,1)$ account for the natural storage depreciation and  transmission line losses respectively. We assume that the uncertainties in power generation and storage device operation are modeled by the disturbance process $w_k$,~$k\in \K$.  
We consider the following constraints
\begin{subequations}
    \begin{align}
        & \text{Storage: } \underline{x} \leq \E[x_{k+1}]= a~ \E[x_k]+\sum_{i=1}^{2}b ^i~\E[u_k^i]-L_k ,\label{eq:CoupleConst}\\
		&\text{Generation: }  \underline{u}^{i}\leq \E[u_k^i] \leq \bar{u}^i,~i=1,2\label{eq:PrivateConst}.
    \end{align}
\end{subequations}
The mixed (coupled) constraint \eqref{eq:CoupleConst} indicates the reserve level of the storage unit, that is, the mean storage level cannot go below $\underline{x}$.
Further, \eqref{eq:PrivateConst} represent the operational constraints of the generators, that is, the mean/expected production level $\E{[u_k^i]}$ of each generator $i=1,2$ cannot go above  $\bar{u}^i$ and below $\underline{u}^i$. The generating units seek to minimize their production costs which are proportional to their generation levels. Further, they try to minimize variance in their generation levels. The generating units  wish not to have high storage levels when they are able to meet the demand, and also wish to reduce the variance of the storage level. We assume there are no terminal costs. So, the cost functional of each generating unit $i=1,2$ is given as 
$J^i =\tfrac{1}{2}\sum_{k\in\K_l}\big(q_k^i\E[(x_k-\E[x_k])^2]+(q_k^i+\bar{q}_k^i)\E[x_k]^2\big)  +\tfrac{1}{2}\sum_{k\in\K_l}\big(r_k^{ii}\E[(u_k^i-\E[u_k^i])^2]+(r^{ii}_k+\bar{r}^{ii}_k)\E[u_k^i]^2\big)$.
 \begin{figure}[h]  
\subfloat[]{\begin{tikzpicture}[scale=0.5,>=latex']
  		\tikzset{every pin/.append style={font=\large}}
  	\begin{groupplot}[
  		group style={
  			group size=1 by 2,
  			horizontal sep=1.25cm,
  		},normalsize]
  		\nextgroupplot[
  		width=0.475\textwidth,
  		height=0.52\textwidth,
  		xmin=0,xmax=140,ymin=-0, ymax=11, xlabel = {time [h] },ylabel={Load and generator production levels},legend pos = north west,
  		grid=both
  		]

		\addplot[color=Red!100, line width=1.0pt, dashed,  very thick] table{DataGSnew/Eu1.dat}; \addlegendentry{$\E[u_k^1]$};
		\addplot[color=RoyalBlue, line width=1.0pt, solid, thin ] table{DataGSnew/u1.dat}; \addlegendentry{$u_k^1$};
  \addplot[color=magenta!100, line width=1.0pt, dashed, very  thick] table{DataGSnew/Eu2.dat}; \addlegendentry{$\E[u_k^2]$};
		\addplot[color=Green, line width=1.0pt, solid, thin ] table{DataGSnew/u2.dat}; \addlegendentry{$u_k^2$};
		  		\addplot[color= RedOrange , line width=1.0pt, solid , thin] table{DataGSnew/L.dat}; \addlegendentry{$L_k$}; 
  \draw[-,black!50,dashed,thick](0,4.5) to (140,4.5)node[xshift=-1cm,above,black]{\large{$\bar{u}^{1}$}};
  \draw[-,black!50,dashed,thick](0,7) to (140,7)node[xshift=-1cm,above,black]{\large{$\bar{u}^{2}$}};
  \draw[-,black!50,dashed,thick](0,1.5) to (140,1.5)node[xshift=-3.5cm,above,black]{\large{$\underline{u}^{1}$}};
  \draw[-,black!50,dashed,thick](0,0.5) to (140,0.5)node[xshift=-3.5cm, yshift=-0.7,above,black]{\large{$\underline{u}^{2}$}};
\end{groupplot}
    \end{tikzpicture}
\label{fig:fig2a}}
\subfloat[]{\begin{tikzpicture}[scale=0.5,>=latex']
	    \tikzset{every pin/.append style={font=\large}}
	    \begin{groupplot}[
	    	group style={
	    		group size=1 by 2,
	    		horizontal sep=1.25cm,
	    	},normalsize]
	    	\nextgroupplot[
	    	width=0.475\textwidth,
	    	height=0.375  \textwidth,
	    	xmin=0,xmax=140,ymin=-1.5, ymax=7, 
	    	ylabel={Storage level },legend pos = north east,grid=both, 
	    	legend style={at={(0.3,0.8)},anchor=west}
	    	]
 
		\addplot[color=Red!100, line width=1.0pt, dashed, very thick] 
            table{DataGSnew/Ex.dat}; \addlegendentry{$\E[x_k]$};
		\addplot[color=RoyalBlue, line width=1.0pt, solid, thin] 
            table{DataGSnew/x.dat}; \addlegendentry{$x_k$};
        \addplot[color=Green, line width=1.0pt, solid, thick] 
            table{DataGSnew/state_diff.dat}; \addlegendentry{$\E[x_{k+1}]-\E[x_k]$};
            \draw[-,black!50,dashed,thick](0,1.5) to (140,1.5)node[xshift=-5.5cm,above,black]{\large{$\underline{x}$}};
  	\nextgroupplot[
  width=0.475\textwidth,
  height=0.175\textwidth,
 xmin=0,xmax=140,ymin=-4, ymax=4, xlabel = {time [h] },ylabel={Disturbance},
 grid=both, legend style={at={(0.7,0.9)},anchor=west}, 
 grid style={line width=.1pt, draw=gray!25}
  ]
  	\addplot[color=RoyalBlue, line width=1.0pt, solid, thin] table{DataGSnew/noise.dat}; \addlegendentry{$w_k$};
\end{groupplot}
      \end{tikzpicture}
  \label{fig:fig2b}
 }
	\caption{Panel (a) depicts time varying load and generator outputs, and panel (b) depicts the battery storage level  and disturbance signal.} 
	\label{fig:Fig2}
\end{figure}
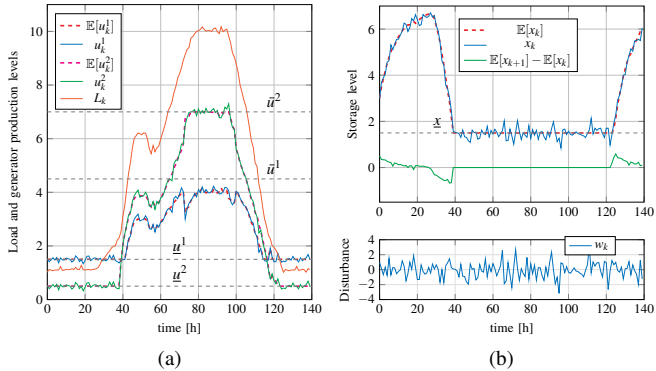

For numerical illustration, we consider the following parameter values: $r_k^{ii}=0.5$, $\bar{r}_k^{ii}=2.5$, $q_k^{i}=3.5$, $\bar{q}_k^{i}=0.5$, $i=1,2, ~k\in \K_l$, $a=0.9$, $b^1=0.90$, $b^2=0.94$, $\underline{x}=1.5$,  $\underline{u}^1=1.5$, $\underline{u}^2=0.5$, $\bar{u}^1=4.5$, $\bar{u}^2=7$, $K=140$,  $\sigma =0.2$, $x_0=3$ (deterministic). We consider the disturbance signal  $w_k,~k\in \K_l$ to be a white Gaussian noise process. For the chosen parameter values, we note that the conditions required in Assumption \ref{ass:Matrixinvertivility} are satisfied. We used the freely
available PATH solver (available at 
\url{https://pages.cs.wisc.edu/~ferris/path.html}) for solving the \texttt{LCP} \eqref{eq:FinalLCP}.
Fig. \ref{fig:fig2a} illustrates 
the time varying load  and   generator production levels. 
We observe that the generators vary their production levels while satisfying the 
generation constraints \eqref{eq:PrivateConst}. Fig. \ref{fig:fig2b} illustrates 
the battery storage levels and the disturbance signal. In particular, we observe that when $\E[x_{k+1}]-\E[x_k] = (a-1)\E[x_k]+b^1\E[u_k^1]+b^2\E[u_k^2]-L_k <0$, the storage unit discharges towards meeting the demand and thereby reaches its reserve level $\underline{x}$, satisfying the mixed coupled constraint \eqref{eq:CoupleConst}.
\section{Conclusion}\label{sec:Conclusion}
We have characterized the solution for a class of linear-quadratic MFTDGs with coupled affine inequality constraints. This involves a multiplier process satisfying implicit complementarity conditions. By reformulating these conditions as a single large-scale linear complementarity problem, we enable computation of these solutions. A numerical example has illustrated our proposed approach. In future work, we aim to generalize the constraint structure to include state and control terms, alongside their mean terms. Further, we also plan to explore the problem in the continuous-time formulation.
\bibliographystyle{IEEEtran}
\bibliography{main} 
\newpage

 \end{document}